\documentclass[a4paper,6pt]{amsart}
\usepackage[T1]{fontenc}
\usepackage[utf8]{inputenc}
\usepackage{amssymb} 
\usepackage{amsthm}
\usepackage{amsmath}
\usepackage{mathrsfs}
\usepackage{graphicx}
\usepackage{tikz,tkz-tab}
\usepackage{url}
\usepackage{xcolor}
\usepackage{enumitem}
\usepackage{bbold}
\usepackage{amssymb,amsmath,graphicx,color}
\newcommand{\cRM}[1]{\MakeUppercase{\romannumeral #1}}

\newtheorem{theorem}{Theorem}[section]
\newtheorem{remarque}[theorem]{Remark}
\newtheorem{proposition}[theorem]{Proposition}
\newtheorem{corollaire}[theorem]{Corollary}
\newtheorem{lemme}[theorem]{Lemma}

\newtheorem{exmp}{Exemple}[section]

\title[Hypercyclicity and compactness in de Branges-Rovnyak spaces]{Hypercyclicity and compactness of co-analytic Toeplitz operators on de Branges-Rovnyak spaces}

\author[Alhajj]{Rim Alhajj}
 \address{Laboratoire Paul Painlev\'e, Universit\'e Lille 1, 59 655 Villeneuve d'Ascq C\'edex, France}
 \email{rim.alhajj@univ-lille.fr}
 


\thanks{The author would like to warmly thank  Emmanuel Fricain for his support and for his
valuable remarks and suggestions which improved the quality of this paper. The author
was supported by the Laboratory CEMPI (Centre Européen pour les Mathématiques, la Physique et leurs interactions).}

\begin{document}
\maketitle
\begin{abstract}
We study the compactness and the hypercyclicity of Toeplitz operators in the de Branges-Rovnyak spaces $\mathcal{ H}(b)$ with co-analytic and bounded symbols on $ \mathbb{D}$. We highlight the fundamental role played by the function $b$ generating the de Branges-Rovnyak space $ \mathcal{ H}(b) $. The characterization of compactness depends wether $b$ is inner function or not and the characterization of hypercyclicity depends wether the function  $\log (1 - |b|)$ is integrable or not.
\end{abstract}

\section{Introduction}

We shall mostly be discussing co-analytic Toeplitz operators $T_{ \bar{ \varphi}}$ with symbol $\bar{ \varphi}$  where $ \varphi \in H^{ \infty}$, that are naturally defined on the de Branges-Rovnyak space into itself. These operators have been introduced by Lotto-Sarason in \cite[Lemma 2.6]{MR1133377}, see also \cite[Section \cRM{2}.7]{MR1289670}. Although some special cases have long ago appeared in literature for $ \varphi \in L^{ \infty} ( \mathbb{T}) $, most notably as standard Toeplitz operators $ T_{  \varphi} :H^2 \rightarrow H^2 $ studied by A. Brown and P. Halmos in the paper \cite{MR0160136} and as the adjoints of truncated Toeplitz operators $A_{ \varphi}^{ \Theta}$ on model spaces $ K_{ \Theta}$ introduced by Sarason in \cite{MR2363975}.

It turns out that de Branges-Rovnyak spaces, which are a family of subspaces  $\mathcal{ H}(b)$ of the Hardy space $H^2$, parametrized by elements $b$ of the closed unit ball of $H^{ \infty}$ are invariant under $T_{ \bar{ \varphi}}$, where $\varphi \in H^{ \infty}$. We shall give the precise definition in section $2$. In general $\mathcal{ H}(b)$ is not closed in $H^2$, but it carries its own norm $||.||_{ \mathcal{ H}(b)} $ making it a Hilbert space. The spaces $ \mathcal{ H}(b)$ were introduced by de Branges and Rovnyak in the appendix of \cite{MR0244795} and further studied in their book \cite{MR0215065}.

The general theory of $\mathcal{ H}(b)$-spaces generally splits into two cases, according to whether $b$ is an extreme point or a non-extreme point of the unit ball of $H^{ \infty}$. The dichotomie $b$ extreme/non-extreme will also greatly appear in this paper. The general idea is that the extreme case has many features that are not far from the case of $b = \Theta$ inner (the classical model space $K_{ \Theta}$),while the non-extreme case has several properties that are similar to the case where $b = 0$ (the Hardy space $H^2$).

%while the non extreme case is more exotic from this point of view.\\
This paper treats two properties related to the restricted Toeplitz operators $ T_{ \bar{ \varphi}}$ on the de Branges-Rovnyak space $ \mathcal{ H}(b) $ where $ \varphi \in H^{ \infty}$. One of these properties is based on the particular operator $X_b = {T_{ \bar{z}}}_{| \mathcal H(b)}$ that plays a central role in the theory and particularly in the model develop by de Branges and Rovnyak. Indeed, it serves as a model for a large class of contractions (see \cite[Theorem 26.16]{MR3617311}). %Moreover, a lot of properties of $\mathcal{ H}(b)$ functions are reflected in the property of this operator.

 %The properties under consideration are compactness and hypercyclicity.
% However most of the results presented are consequences of recent work on standard Toeplitz operators and truncated Toeplitz operators.\\ 

The first one concerns the compactness of the Toeplitz operator ${T_{\overline \varphi}}_{| \mathcal H(b)}$, %and splits according to whether $b$ is extreme or not, % In the case where $b$ is non-extreme, we shall prove the non-compactness of the Toeplitz operator ${T_{\overline \varphi}}_{| \mathcal H(b)}$ (except for the trivial case where $\varphi = 0$), since Brown-Halmos \cite{MR0160136} has already shown the non-compactness of the standard Toeplitz operator $ T_{ \bar{ \varphi}} : H^2 \rightarrow H^2$, for every $ \varphi \in L^{ \infty} ( \mathbb{T})$ ( also except for the trivial case where $\varphi = 0$). Whereas for the extreme case.
More precisely, given an element $b$ of the closed unit ball of $ H^{ \infty}$, that is extreme or non-extreme, we are interested in finding a necessary and sufficient condition for the operator ${T_{\overline \varphi}}_{| \mathcal H(b)}$ to be compact. Let us mention that Brown and Halmos \cite{MR0160136} have shown that there is no compact Toeplitz operators $ T_{  \varphi} :H^2 \rightarrow H^2 $ with symbol $ \varphi \in  L^{ \infty} ( \mathbb{T})$ (except the trivial case where $\varphi = 0$). It is not surprising that in the case where $b$ is non-extreme, we reach a similar result for the restricted Toeplitz operator ${T_{\overline \varphi}}_{| \mathcal H(b)}$ with $ \varphi \in H^{ \infty}$.  
%a result in terms of the compactness of the standard Toeplitz operator $ T_{ \varphi} : H^2 \rightarrow H^2$  with $ \varphi \in L^{ \infty} ( \mathbb{T})$ where they proved that there is no compact op

On the other hand, Ahern and Clark \cite{MR0264385} studied the compactness of truncated Toeplitz operators $A_{ \varphi}^{ \Theta}$, with continuous symbol $ \varphi$ on $ \mathbb{T}$ and they got a necessary and sufficient condition based on the image of the spectrum of $ \Theta$ intersected with the unit circle. Recently, Garcia, Ross and Wogen \cite{MR3203060} have found this result with another proof. Using their method, we show that when $b$ is an extreme point of the closed unit ball of $H^{ \infty}$ which is not inner, then the only compact Toeplitz operators ${T_{\overline \varphi}}_{| \mathcal H(b)} $, where $ \varphi \in H^{ \infty} \cap C( \mathbb{T}),$ corresponds to the trivial case $ \varphi = 0.$ Combining our two results, we see that when $b$ is not an inner fonction, there are no compact Toeplitz operators ${T_{\overline \varphi}}_{| \mathcal H(b)} $ with $\varphi \in H^{ \infty} \cap C( \mathbb{T})$ (expect when $\varphi = 0$).
 %We show that their method can be extended in order to obtain similar results for the general de Branges-Rovnyak spaces $ \mathcal{ H}(b) $ where $b$ is an extreme point of the unit ball of $H^{ \infty}$ and $ \varphi \in H^{ \infty} \cap C( \mathbb{T})$ (since the restriction of co-analytic symbols is necessary if we want that $T_{ \bar{ \varphi }} $ maps $ \mathcal{H}(b)$ into itself). Also we notice that in this case, the Toeplitz operator ${T_{\overline \varphi}}_{| \mathcal H(b)}$ can be compact only when $b$ is inner (i.e. $b=\Theta$, since an inner function is extreme in particular).\\
 
Our second problem is related to the hypercyclicity of the restricted Toeplitz operator ${T_{\overline \varphi}}_{| \mathcal H(b)}$. Godefroy and Shapiro, using their Criterion \cite{MR1111569}, proved that for $ \varphi \in H^{ \infty}$, then $ T_{ \bar{ \varphi }} : H^2 \rightarrow H^2$ is hypercyclic if and only if $ \varphi$ is non constant and $ \varphi ( \mathbb{D}) \cap \mathbb{T} \neq \emptyset $. On the other hand, it is obvious that there are no hypercyclic Toeplitz operators with analytic symbols. For general symbols $ \varphi \in L^{ \infty} ( \mathbb{T})$, few results are known; see a recent paper of A.Baranov and A. Lishanskii \cite{MR3544864} and a paper of Shkarin who studied the case where $ \varphi (z) = a \bar{z} + b + cz.$ We succeed to extend the result of Godefroy and Shapiro in our context of $\mathcal{H}(b) $ spaces, when $ b $ is non extreme. Whereas, when $b$ is extreme, we give a necessary condition for the operator ${T_{\overline \varphi}}_{| \mathcal H(b)}$ to be non-hypercyclic. This condition is based on the point spectrum of the operator. It also turns out that this necessary condition is not sufficient.\\

The structure of the paper is the following. After a preliminary section with generalities about Branges-Rovnyak spaces, Toeplitz operators and definitions of hypercyclic and frequently hypercyclic operators, %with some examples of hypercyclic Toeplitz operators in section 2. 
we discuss compactness properties on the restricted Toeplitz operators ${T_{\overline \varphi}}_{| \mathcal H(b)}$ in Section 3. The last section is dedicated to the hypercyclicity of ${T_{\overline \varphi}}_{| \mathcal H(b)}$.\\

 % In Section 3, we show that their methods in 
 %however we reach results that are respectively similar to the compactness of truncated Toeplitz operator and it was the first result that dates from 1970. In [1, Section 5] with continuous symbol, when the hypothesis that $ \varphi$ belongs to $H^{ \infty}$ is replaced by the assumption that $ \varphi$ is $H^{ \infty} \cap C( \mathbb{T}),$ we  
%Surprisingly enough, the first one about compactness of the truncated Toeplitz operator dates from 1970. In [1, Section 5] with continuous symbol, one proves the following theorem.

%In the original definition of an $\mathcal{ H}(b)$-space, we use two Toeplitz operators. Hence, no doubt, they are extremely important in this context.\\

\section{Preliminaries}

\subsection{Toeplitz operators and de Branges-Rovnyak spaces}\label{section 2.1}

We first recall some basic facts on Toeplitz operators on the Hardy space $ H^2$ of the open unit disc $ \mathbb{D}= \{  z \in \mathbb{C} : |z| < 1  \}$.

Given $\varphi \in L^{ \infty}( \mathbb{T}) =L^{ \infty}( \mathbb{T},m) $ where $\mathbb{T} = \partial{ \mathbb{D}} $ and $m$ is the normalized lebesgue measure on $\mathbb{T}$, the corresponding Toeplitz operator $T_{  \varphi} : H^2 \rightarrow H^2 $ is defined by
$$T_{  \varphi} f := P_+ (\varphi f ) \hspace{0.5 cm} (f \in H^2 ),$$
where $P_+ : L^2 ( \mathbb{T}) \rightarrow H^2$ denotes the othogonal projection of $L^2 ( \mathbb{T}) = L^2( \mathbb{T},m)$ onto $H^2$. Clearly $T_{  \varphi}$ is a bounded operator on $H^2$ with $ ||T_{  \varphi} || = ||  \varphi||_{L^{ \infty} ( \mathbb{T})}$, moreover it is compact if and only if   $\varphi = 0$ (Brown--Halmos, \cite{MR0160136}). If $ \varphi \in  H^{ \infty}$ the algebra of the analytic and bounded functions on $ \mathbb{D}$, then $T_{  \varphi}$ is simply the operator of multiplication by $ \varphi$ and its adjoint is $T_{  \bar{ \varphi}}$. Consequently, if $ \varphi, \psi \in H^{ \infty} $, then $ T_{  \bar{ \varphi}} T_{  \bar{ \psi}} = T_{  \bar{ \varphi} \bar{ \psi}} = T_{  \bar{ \psi}}T_{  \bar{ \varphi}}$.\\

If $ \varphi \in L^{ \infty} (\mathbb{T})$ satisfies $||\varphi||_{ \infty} \leq 1$, then, the corresponding Toeplitz operator $T_{ \varphi}$ is a contraction on the Hilbert space $H^2$. The associated de Branges-Rovnyak space $\mathcal H(T_{ \varphi} )$ is defined by
$$ \mathcal H(T_{ \varphi})= ( I - T_{ \varphi} T_{  \bar{ \varphi}})^{1/2}H^2.$$ 
For simplicity, we denote the complementary space $\mathcal H(T_{ \varphi} $) by $\mathcal H( \varphi)$ (see \cite[Section 17.3]{MR3617311}). Our main concern is when $ \varphi$ is a nonconstant analytic function in the closed unit ball of $H^{ \infty}$. In this case, by tradition, we use $b$ instead of $\varphi$. Therefore, the definition of an $\mathcal H(b)$-space uses the defect of the contraction $T_{b}$ \cite{MR3617311}. Hence, no doubt, the Toeplitz operators are extremely important in this context.\\ 

Here we recall an alternative and equivalent definition based on reproducing kernel. For every function $b$ in the closed unit ball of $ H^{ \infty}$, we associate the de Branges-Rovnyak space $\mathcal H(b)$ defined as the Hilbert space of analytic functions on $ \mathbb{D}$ whose reproducing kernel is given by
$$ k_{ \lambda}^b ( z) = \frac{1 - \overline{b( \lambda)}b(z)}{1 - \bar{ \lambda}z},\qquad \lambda,z\in\mathbb D.$$\\
That is,
 $$f( \lambda ) =  <f,k_{ \lambda}^b >_b, \forall f \in { \mathcal{H} }(b) , \forall \lambda \in \mathbb{D}.$$\\
For $ b = 0 $, we see that $ k_{ \lambda}^b$ coincides with $k_{ \lambda}$ the reproducing kernel of $ H^2$, given by $ k_{ \lambda} (z) = ( 1 - \bar{ \lambda} z )^{ -1}$, whence $ { \mathcal{H}}(0) = H^2$.\\
More generally when $ ||b||_{ \infty} < 1 $, then $\mathcal H(b)$ coincides with the Hardy space $H^2$ with an equivalent norm.\\
For $ b = \Theta$, with $\Theta$ an inner function (that is a function in the closed unit ball of $ H^{ \infty}$ such that $ | \Theta ( \zeta)|= 1 $  almost everywhere on $\mathbb T=\partial\mathbb D$), the space ${ \mathcal{H} }(\Theta)$ is a closed subspace of $ H^2$, and we have\\ 
 $$ { \mathcal{H} }(\Theta) =( \Theta  H^2 )^{ \perp}:=\{f\in H^2:\langle f,\Theta g\rangle_2=0,\forall g\in H^2\}.$$
The space ${ \mathcal{H} }(\Theta)$ is also called the model space denoted by $ K_{ \Theta} ={ \mathcal{H} }(\Theta) $. By Beurling's theorem, the space $K_{ \Theta}$ correspond to the lattice of closed, non trivial, invariant subspaces for the backward shift operator $ S^* = T_{ \bar{z}} $ on $ H^2$.\\
 %Beurling was able to show that these model spaces are the closed subspaces of $ H^2 $ invariant by the backward shift $ S^* $. Recall that $ Sf(z) = zf(z) $ and its adjoint $ S^* f(z) = (f (z) -f (0)) / z $.
 
In the general case, the spaces $ \mathcal{ H }(b) $ are Hilbert spaces that are contained contractively in $ H^2 $.
Moreover, it's well-known that there are relations between the inner products of $ \mathcal{ H }(b) $ and its cousin $ \mathcal{ H }( \bar{b}) $ since these relations are special cases of the Lotto–Sarason theorem \cite[Theorem 16.18 and corollary 16.19]{MR3617311}. For further reference, we restate this result below.
\begin{theorem}[\cite{MR3617311}, Theorem 17.8]\label{thm 1.1}
%$\cite[\hspace{0.1 cm}  Theorem \hspace{0.1 cm} 17.8]{MR3617311}$\label{thm 1.1}
Let $ f \in H ^2$. Then $ f \in \mathcal{ H }(b)$ if and only if $ T_{ \bar{b} }f \in \mathcal{ H }( \bar{b})$ and
$$    <f_1, f_2 >_b = <f_1,f_2>_2 + <T_{ \bar{b} }f_1,T_{ \bar{b} }f_2 >_{ \bar{b}} , \hspace{0.5cm} ( f_1, f_2 \in \mathcal{ H }(b)).          $$
\end{theorem}
It is now a well-known fact that the general theory of $\mathcal{ H }(b)$-spaces splits into two cases, according to whether $b$ is an extreme point or a non-extreme point of the unit ball of $ H^{ \infty}$ (recall that, according to De Leew-Rudin's Theorem, $b$ is a non-extreme point of the closed unit ball of $ H^{ \infty}$ if and only if $ \log (1 - | b |) \in L^1 ( \mathbb{T})$, in particular every inner function $b = \Theta$ is an extreme point).

For example, $\mathcal{ H }(b)$ contains all the reproducing kernels of the Hardy space $H^2 $ if and only if $b$ is a non extreme point of the closed unit ball of $H^{ \infty}$ (see \cite[ Theorem 23.23 and corollary 25.8]{MR3617311}).
\\

Furthermore from the above characterization of a non-extreme point it follows that, if $b$ is non-extreme, then there is an outer function $a$ such that $a(0) > 0$ and $|a|^2 + |b|^2 = 1 \hspace{0.1cm} a.e. \hspace{0.1cm} on \hspace{0.1cm} \mathbb{T}$ \cite{MR1289670}. The function $a$ is uniquely determined by $b$. We shall call $(b, a)$ a pair.The following result gives a useful characterization of $ \mathcal{H}(b)$ in this case.
\begin{theorem}[\cite{MR3617311}, Theorem 23.8]\label{thm 1.2}
Let $b$ be a non-extreme point of the closed unit ball of $ H^{ \infty}$, let $(a,b)$ be a pair and let $f \in H^2$. Then $f  \in { \mathcal{H}}(b)$ if and only if $T_{ \bar{b}} f  \in T_{ \bar{a}} (H^2 )$. In this case, there exists a unique function $f^+  \in H^2$ such that $T_{ \bar{b}}  f = T_{ \bar{a}} f^+$, and
$$ || f||_b^2 = ||f ||_2^2 + ||  f^+||_2^2 .$$
In particular, for $f,g \in \mathcal{H}(b)$ there exists a unique $ f^+ , g^+ \in H^2$ such that\\
$$  <f,g>_b = <f,g>_2 + <f ^+, g^+>_2.   $$
\end{theorem}

An important operator in the theory of model spaces is the compression of Toeplitz operators on $ K_{ \Theta} $: for $ \varphi \in L ^{ \infty} $ and $ \Theta $ an inner function, one defines the truncated Toeplitz operator $ A_{ \varphi} ^ {\Theta} $ by
$$
 \begin{array}{ccccc}
 A_{ \varphi}^{ \Theta} & : & K_{ \Theta} & \to & K_{ \Theta} \\
 & & f & \mapsto &   A_{ \varphi}^{ \Theta}f = \varphi ( M_{ \Theta})(f)  = { \mathbf{P}}_{ \Theta}(T_{ \varphi}f),
\end{array}
$$
with ${ \mathbf{P}}_{ \Theta}$ the orthogonal projection of $ H^2$ to $ K_{ \Theta}$. It turns out that when $\varphi$ is in $H^\infty$, then $K_\Theta$ is invariant by $T_{\overline{\varphi}}$ and the adjoint of the truncated Toeplitz operator with symbol $\varphi$ is  $(A_{\varphi}^\Theta)^*= T_{\overline{\varphi}}|K_\Theta$ \cite[ Section 14.7]{MR3617311}. More generally, for $ \varphi \in H^{ \infty } $, ${ \mathcal{H} }(b)$ is invariant by $ T_{\bar{ \varphi}}$  and $||T_{\bar{ \varphi}}||_{L( { \mathcal{H} }(b) ) } \leq || \varphi||_{ \infty}  $ \cite{MR3617311}. In particular, ${ \mathcal{H} }(b)$ is invariant by the backward shift $ S^* = T_{ \bar{z}}$.\\
 
Ahern and Clark \cite{MR0264385} have given a necessary and sufficient condition for the truncated Toeplitz operator $ A_{ \varphi} ^ {\Theta} $ to be compact, when the symbol $ \varphi$ is continuous on the boundary. See also an alternative proof by Garcia-Ross-Wogen in \cite{MR3203060}. The characterization of Ahern-Clark involves the notion of the spectrum of an inner function.

Recall that the spectrum of a function $b$ in the closed unit ball of $ H^{ \infty}$ \cite[Section 5.2 and 22.6]{MR3617311}, denoted by $ \sigma (b)$ is defined as follows\\
$$  \sigma(b)= \left\lbrace  \zeta \in \overline{ \mathbb{D}} : \lim\limits_{ \substack{ z \rightarrow \zeta \\ z \in \mathbb{D}} } \inf | b(z)| < 1 \right\rbrace . $$

A generalization of Livsic-Moeller's result shows that $b$ and every element in ${ \mathcal{H} }(b)$ can be analytically continued accross any arc $I \subset \mathbb{T} \setminus clos( \sigma (b))$, and $|b| = 1$ on $I$ \cite[Theorem 20.13]{MR3617311}.\\

In particular if $b = \Theta$ is a non constant inner function, and since $\Theta$ is unimodular a.e. on $\mathbb{T}$ then\\
$$ \sigma( \Theta ) = \left\lbrace \zeta \in \overline{ \mathbb{D}} : \lim\limits_{ \substack{ z \rightarrow \zeta \\ z \in \mathbb{D}} } \inf | \Theta(z)| = 0 \right\rbrace = clos (Z( \Theta)) \cup supp( \nu),$$\\
where  $Z( \Theta)= \{ \lambda \in \mathbb{D} : \Theta ( \lambda) = 0  \}$ and $\nu$ is the measure representing the singular part of $\Theta$.\\

Now Ahern and Clark's result says: 
\begin{theorem}[Ahern-Clark, \cite{MR0264385}]
Let $ \varphi \in C( \mathbb{T})$, then $A_{ \varphi}^{ \Theta}$ is compact if and only if $ \varphi_{ | { { \sigma}( \Theta) \cap \mathbb{T} } } = 0 .  $
\end{theorem} 
\subsection{Hypercyclic and frequently hypercyclic operators}

Let $X$ be a complex infinite-dimensional separable Banach space. An operator $ T \in L(X)$ is said to be hypercyclic if there is some vector $x \in X$ such that the orbit
$$ O(x,T) := \{ T^n(x); n \in \mathbb{N}    \}$$ 
is dense in $X$. Such a vector $x$ is said to be hypecyclic for $T$, and the set of all hypercyclic vectors for $T$ is denoted by $HC(T)$.\\

Moreover we say that $T$ is frequently hypercyclic, if there exists a vector $x \in X$ such that for every non-empty open subset $U$ of $X$, the set $ N(x,U) = \{ n \geq 0 ;T^n(x) \in U \}$ of instants when the iterates of $x$ under $T$ visit $U$ has positive lower density, i.e. 
$$ \underline{dens}( N(x,U) )=  \lim\limits_{N \rightarrow \infty} \inf \frac{card ( N(x,U) \cap [ 1, N]) }{N} > 0.           $$

We refer the reader to the recent book \cite{MR2533318} for more information on these topics.

Frequent hypercyclicity is a much stronger notion than hypercyclicity, and some operators are hypercyclic without being frequently hypercyclic: an example is the Bergman backward shift \cite{MR2159459}. \\
 
Let us complete this section by recalling two criterions for hypercyclicity and frequent hypercyclicity that we will use to study the hypercyclicity properties of the Toeplitz operator ${T_{\overline \varphi}}_{| \mathcal H(b)}$.\\

We start with the Godefroy-Shapiro Criterion \cite{ MR1111569}, according to which a bounded operator having a large supply of eigenvectors associated to eigenvalues of modulus strictly larger than 1 and strictly smaller than 1 is hypercyclic.\\
\begin{theorem}[Godefroy-Shapiro Criterion, \cite{ MR1111569}]\label{thm 2.4}
Let $ T \in L(X)$ . Suppose that $  \bigcup_{ | \lambda | < 1} Ker( T - \lambda ) $ and $  \bigcup_{ | \lambda | > 1} Ker( T - \lambda )$ both span a dense subspace of $X$. Then $T$ is hypercyclic.\\
\end{theorem}
%That the behaviour of the eigenvectors of an operator has an influence on its hypercyclicity properties was first discovered by Godefroy and Shapiro in \cite{MR1111569}: their work deals with eigenvectors associated to eigenvalues of modulus strictly larger than 1 and strictly smaller than 1. 
Then it was shown by S.Grivaux that an operator $T$ which has "sufficiently many" eigenvectors associated to eigenvalues of modulus $1$ in the sense that these eigenvectors are perfectly spanning, then $T$ is automatically frequently hypercyclic.\\
\begin{theorem}[S.Grivaux, \cite{MR2975340}]\label{thm 2.5}
Let $T$ be a bounded operator on $X$. Suppose that there exists a sequence $(u_i)_{i \geq 1}$ of vectors of $X$ having the following properties:
\begin{enumerate}
\item For each $i \geq 1$, $u_i$ is an eigenvector of $T$ associated to an eigenvalue $ {\lambda}_i$ of $T$ with $|{\lambda}_i|= 1$ and the ${\lambda}_i$’s all distinct;
\item $span[u_i;i \geq 1]$ is dense in $X$;
\item  For any $i \geq 1$ and any $ \varepsilon > 0$, there exists an $ n \neq i$ such that $||u_n - u_i|| < \varepsilon $.
\end{enumerate}
Then $T$ is frequently hypercyclic (in particular hypercyclic).
\end{theorem}

It is also a natural question, given a family of hypercyclic operators to ask if they have a common hypercyclic vector. The following result gives a sufficient condition for a family of multiple of an operator to have a dense $G_{ \delta}$-set of common hypercyclic vectors.
\begin{theorem}[ Shkarin, \cite{MR2557957}]\label{thm 1.7}
Let $X$ be a separable  Fréchet space, $ T \in L(X)$, $ 0 \leq a < c \leq \infty$. Assume also that for all $ \alpha , \beta \in \mathbb{R}$ such that $ a < \alpha < \beta < c$ there exists a dense subset $E$ of $X$ and a map $ S: E \longrightarrow E$ such that $ TSx = x$, $ { \alpha }^{ -n}T^n x \rightarrow 0$ and $ { \beta }^n S^n x \rightarrow 0$ for each $x \in E$. Then
$$ \cap HC( \lambda T : c^{-1} <  | \lambda |<  a^{ -1} ),$$
is a dense $ G_{ \delta}$-set in $X$.
\end{theorem}

We finish by given particular hypercyclic Toeplitz operators.

Rolewicz's result \cite{MR2533318} in 1960, says that the operator $  \lambda S^*= T_{ \lambda \bar{z}}: H^2 \rightarrow H^2$ for every $ \lambda \in \mathbb{C}, \hspace{0.1cm} | \lambda| > 1$, is hypercyclic, and it was shown using Kitai's Criterion (a particular case of the Hypercyclity Criterion) \cite{MR884467}. This result of Rolewicz was generalized by Godefroy-Shapiro \cite{MR1111569} in 1991.
\begin{theorem}[Godefroy-Shapiro]\label{thm 2.6} 
Let $ \varphi \in H^{ \infty}$. The operator $T_{ \bar{ \varphi} } : H^2 \rightarrow H^2$ is hypercyclic if and only if $\varphi$ is non-constant and $ \varphi ( \mathbb{D}) \cap \mathbb{T} \neq \emptyset$. 
\end{theorem}

\section{Compactness of ${T_{\overline \varphi}}_{| \mathcal H(b)}$}

The compactness property of the operators ${T_{\overline \varphi}}_{| \mathcal H(b)}$ will depend greatly on whether $ b $ is an extreme point of the closed unit ball of $ H^{ \infty} $ or not.\\
The non-extreme case is similar to what happens for $ H^2$. In other words, when $ b $ is non-extreme, ${T_{\overline \varphi}}_{| \mathcal H(b)}$ with $ \varphi \in H^{ \infty}$ is never compact (except for the trivial case where $ \varphi = 0 $).\\
While for the extreme case, compactness results of the operator ${T_{\overline \varphi}}_{| \mathcal H(b)}$ are obtained only with $ \varphi \in H^{ \infty} \cap C( \mathbb{T})$ and they depend on wether $b$ is inner or not.\\

\begin{theorem}\label{thm 2.1}
Let $b$ be a non-extreme point of the closed unit ball of $ H^{ \infty}$ and let $ \varphi \in H^{ \infty}$. Then the operator\\
$$\begin{array}{ccccc}
T_{ \overline{ \varphi}}  & : & { \mathcal{H} }(b) & \to & { \mathcal{H} }(b) \\
 & & f & \mapsto & P_+( \overline{ \varphi} f) \\
\end{array}.$$
is compact if and only if $ \varphi = 0 $.
\end{theorem}
\begin{proof}
Let $a$ be the unique outer function such that $(a,b)$ is a pair.\\
\\
Suppose that $  T_{ \bar{ \varphi} }  $ is compact in ${ \mathcal{H} } (b)$. Notice that for every $ ({ \lambda}_n)_n \subset \mathbb{D} $ such that $  | { \lambda}_n| \rightarrow 1 $, the sequence $ ( \frac{k_{ \lambda_n}}{ || k_{ \lambda_n} ||_b})_n$ converges weakly to $0$ in ${ \mathcal{H} } (b)$.\\
 %$$ \forall ({ \lambda}_n)_n \subset \mathbb{D} ; \lim\limits_{ n \rightarrow \infty }| \lambda_n |= 1, \frac{k_{ \lambda_n}}{ || k_{ \lambda_n} ||_b}   \rightharpoonup 0 \hspace{0.1cm} in \hspace{0.1cm}  { \mathcal{H}}(b),$$\\
 \\
 Indeed, let $ f \in { \mathcal{H}}(b)$ such that $f$ and $f^+ \in H^{ \infty} $. Recall that $f^+$ is defined in Theorem \ref{thm 1.2}. Then, using that
  $$ T_{ \bar{b}}k_{ \lambda_n} = \overline{b( { \lambda}_n)} k_{ \lambda_n}= T_{ \bar{a}} \left(\frac{ \overline{b( { \lambda}_n)}}{ \overline{a( { \lambda}_n)}} k_{ \lambda_n} \right),$$\\
we see that $$k_{ \lambda_n}^+ = \frac{ \overline{b( { \lambda}_n)}}{ \overline{a( { \lambda}_n)}} k_{ \lambda_n}.$$
Whence by Theorem  \ref{thm 1.2}, we have 
 \begin{eqnarray*}
  < f, \frac{k_{ \lambda_n}}{ \mid \mid k_{ \lambda_n} \mid\mid_b} >_b &=& <f,\frac{k_{ \lambda_n}}{ \mid \mid k_{ \lambda_n} \mid\mid_b} >_2 +<f^+, \frac{k_{ \lambda_n}^+}{ \mid \mid k_{ \lambda_n}||_b }   >_2 \\
  &=& \left( f( { \lambda}_n)+ \frac{b( { \lambda}_n) }{a( { \lambda}_n)} f^+ ( { \lambda}_n ) \right) \frac{1}{||k_{ { \lambda}_n}||_b} .\\
\end{eqnarray*}  
On the other hand, it is known that in the non-extreme case:
$$ ||{k_{ \lambda}}_n||_b^2 = \frac{1}{ 1 - |{ \lambda}_n|^2} \left( 1 + \frac{ |b( { \lambda}_n)|^2}{|a( { \lambda}_n)|^2}  \right)  $$
(see \cite[Corollary 23.25]{MR3617311}). Hence
$$ \frac{|b( { \lambda}_n)|^2}{|a( { \lambda}_n)|^2  ||k_{ { \lambda}_n}||_b^2 } = \frac{ (1 - |{ \lambda}_n|^2)|b( { \lambda}_n)|^2 }{|a( { \lambda}_n)|^2+ |b( { \lambda}_n)|^2} \leq 1 - |{ \lambda}_n|^2.  $$
Using this inequality and the inequality 
$$\mid \mid k_{ \lambda_n} \mid\mid_b^2  \hspace{0.1cm} \geq \frac{1}{1 - |{ \lambda}_n|^2},$$
we deduce that 
\begin{eqnarray*}
|  < f, \frac{k_{ \lambda_n}}{ || k_{ \lambda_n} ||_b} >_b | &\leq & \left( |f( \lambda_n )| + | f^+ ( \lambda_n )| \right) \sqrt{1 - |{ \lambda}_n|^2}\\
 & \leq & \left( ||f||_{ \infty} + || f^+ ||_{ \infty} \right)\sqrt{1 - |{ \lambda}_n|^2} \rightarrow 0 \hspace{0.5 cm} as \hspace{0.5 cm} n \rightarrow \infty .\\
\end{eqnarray*}

Furthermore, the set $\lbrace  f \in { \mathcal{H}}(b)  ; f \& f^+ \in H^{ \infty } \rbrace  $ is dense in ${ \mathcal{H}}(b),             $ since $  \lbrace k_{ \lambda}^b ;  \lambda \in \mathbb{D}  \rbrace  \subset  \lbrace f \in { \mathcal{H}}(b)  ; f \& f^+ \in H^{ \infty } \rbrace $.
 Indeed, for every $\lambda \in \mathbb{D}$ ,
  $$(k_{ \lambda}^b )^+ =(k_{ \lambda} - \overline{b( \lambda)} b k_{ \lambda})^+ = k_{ \lambda}^+ - \overline{b( \lambda)} (b k_{ \lambda})^+ = \overline{b( \lambda )} a k_{ \lambda} \in H^{ \infty },$$
  with
  $$ (b k_{ \lambda})^+ = \frac{k_{ \lambda}}{\overline{a( \lambda)}} - a k_{ \lambda} $$ (see \cite[Theorem 23.23]{MR3617311}).
 %Since,
 %$$ T_{\bar{b}}( bk_{ \lambda}) = P_+ ( |b|^2 k_{ \lambda}) = P_+ ( ( 1-|a|^2) k_{ \lambda}) =k_{ \lambda} - T_{ \bar{a}} (a k_{ \lambda}) = T_{ \bar{a} } ( \frac{k_{ \lambda}}{\overline{a( \lambda)}} - a k_{ \lambda}) $$ \\
\\
Then, for every $f \in  { \mathcal{H}}(b)$, and every $ \varepsilon > 0$, there exists  $g \in \mathcal{H}(b) \cap H^{ \infty}$ with $g^+ \in H^{ \infty} $ such that $ ||f-g||_b \leq \frac{ \varepsilon}{2}.$\\
\\
Using previous computations we then see that there exists $ N_0 \in \mathbb{N}$ such that $ n \geq N_0 \Rightarrow | < g,\frac{k_{ \lambda_n}}{ || k_{ \lambda_n} ||_b} >_b | \leq \frac{ \varepsilon}{2} $. Hence for $ n \geq N_0$,
\begin{eqnarray*}
 | < f,\frac{k_{ \lambda_n}}{ || k_{ \lambda_n} ||_b}>_b | & \leq & | < f-g,\frac{k_{ \lambda_n}}{ || k_{ \lambda_n}||_b} >_b | + | < g,\frac{k_{ \lambda_n}}{ || k_{ \lambda_n} ||_b} >_b | \\
& \leq & ||f -g||_b +      | < g,\frac{k_{ \lambda_n}}{ || k_{ \lambda_n} ||_b} >_b | \leq  \varepsilon. \\
 \end{eqnarray*}
 Hence, the sequence $ ( \frac{k_{ \lambda_n}}{ || k_{ \lambda_n} ||_b})_n$ converges weakly to $0$ in ${ \mathcal{H} } (b)$.\\
 \\
Now by compactness of ${T_{\overline \varphi}}_{| \mathcal H(b)}$ it follows that for every $ ({ \lambda}_n)_n \subset \mathbb{D} $ such that $  | { \lambda}_n| \rightarrow 1 $,
\begin{equation}\label{eq 1}
 || T_{ \overline{ \varphi} } \frac{k_{ \lambda_n}}{ \mid \mid k_{ \lambda_n} \mid\mid_b} ||_b \rightarrow 0 \hspace{0.5 cm} as \hspace{0.5 cm} n \rightarrow \infty .
 \end{equation}
But $ T_{ \overline{ \varphi} } k_{ \lambda_n} = \overline{ \varphi( \lambda_n)} k_{ \lambda_n}.$
Thus $ \eqref{eq 1}$ is equivalent to
$$ \forall ({ \lambda}_n)_n \subset \mathbb{D} ; \lim\limits_{ n \rightarrow \infty }| \lambda_n |= 1, | \varphi ( \lambda_n ) |  \rightarrow 0 . $$\\
Which implies that $ \varphi = 0.$
\end{proof}
%By observing the proof of the previous Theorem \eqref{thm 2.1}, we preceive that the sequence $ ( \frac{k_{ \lambda_n}}{ || k_{ \lambda_n} ||_b})_n$ belongs to $ \mathcal H(b)$ since $b$ is non-extreme. Since otherwise these reproducing kernels $k_{ \lambda_n}$ do not necessarly belong to $ \mathcal H(b)$ (already noted in the preliminaries).\\
%Contrary to the case where $b$ is non-extreme, it appears that in the extreme case with $ \varphi \in  C( \mathbb{T}) \cap H^{ \infty}$,  and only when $b$ is inner, one can obtain compact operators.\\

The proof of Theorem \ref{thm 2.1} obviously doesn't work in the case when $b$ is an extreme point of the closed unit ball of $H^{ \infty}$, since in that case, the Cauchy kernels $k_{ \lambda}$ do not belong to $\mathcal{H}(b)$ when $b( \lambda)  \neq 0, \hspace{0.1 cm} \lambda \in \mathbb{D}. $ Nevertheless, with additional assumption that $b$ is not inner and for symbols $ \varphi \in H^{ \infty} \cap C( \mathbb{T}) $ we will get a similar result in the extreme case.\\
 
We begin by studying the compactness of the general operator $T_{ \overline{ \varphi}}:{ \mathcal{H} }(b) \rightarrow H^2  $ with $ \varphi \in  C( \mathbb{T})$, using the same technique used by Garcia, Ross and Wogen \cite{MR3203060} to prove the Ahern-Clark result on compactness of $ A_{ \varphi}^{ \Theta}$. 
%and from which we obtain the compactness of the desired operator $T_{ \overline{ \varphi}}:{ \mathcal{H} }(b) \rightarrow { \mathcal{H} }(b)  $ where $ \varphi$ is obviously in  $H^{ \infty} \cap C( \mathbb{T})$.\\
\begin{theorem}\label{thm 2.2}
 Let $b$ be an extreme point of the closed unit ball of $ H^{ \infty}$ and let $ \varphi \in  C( \mathbb{T})$. Then the operator,\\ 
$$\begin{array}{ccccc}
T_{ \overline{ \varphi}}  & : & { \mathcal{H} }(b) & \to & H^2 \\
 & & f & \mapsto & P_+( \overline{ \varphi} f) \\
\end{array},$$\\
is compact if and only if $ \varphi_{ | { { \sigma}(b) \cap \mathbb{T} } } = 0   .$
\end{theorem}
\begin{proof}
$ ( \Leftarrow ) $ Suppose that $ \varphi_{ | { { \sigma}(b) \cap \mathbb{T} } } = 0   .$
Let $ \varepsilon > 0 $ and pick $  \psi \in C( \mathbb{T} ) $; $ \psi = 0$ on an open set containing $ clos({ \sigma}(b) \cap \mathbb{T}) $ and $ || \psi - \varphi ||_{ \infty } < \varepsilon $. Since $  ||T_{ \overline{ \varphi}} - T_{ \overline{ \psi}}||_{ L({ \mathcal{H} }(b), H^2)} \leqslant || \psi - \varphi ||_{ \infty }  < \varepsilon,$ it suffices to show that $ T_{ \overline{ \psi}} : { \mathcal{H} }(b) \rightarrow H^2 $ is compact.\\
\\
Let $ K = \overline { \psi^{ -1} ( \mathbb{C} \setminus \{ 0 \} )}$ then $ K \subset \mathbb{T} \setminus clos ({ \sigma} (b))$. And consider $ (f_n)_n$ a sequence of ${ \mathcal{H} }(b)$ which tends weakly to zero.

We know that each $ f_n \in { \mathcal{H}}(b)$ has an analytic continuation across $K$ (already mentioned in the preliminaries) from which it follows that
$\forall \zeta \in K , f_n( \zeta)= < f_n , k_{ \zeta}^b >_b \rightarrow 0,$ where\\
$$  k_{ \zeta}^b (z)= \frac{1 - \overline{b( \zeta)} b(z)}{1 - \bar{ \zeta}z},
$$
(see \cite[Theorem 21.1]{MR3617311} 
Since $b$ is analytic on a neighborhood of the compact set $K$ we obtain\\
\begin{equation}\label{eq 4}
 \forall \zeta \in K , | f_n( \zeta ) | = | < f_n , k_{ \zeta}^b >_b | \leq ||f_n||_b || k_{ \zeta}^b ||_b \leq c \underset{ \zeta \in K }{\sup} \sqrt{b' ( \zeta )} < \infty \hspace{0.5cm}.
 \end{equation}
%It suffices to show that $ {T_{ \overline{ \psi}} }_{| { \mathcal{H}}(b) }$ is compact, since for every $n \in \mathbb{N}$,\\
%\begin{eqnarray*}
% || T_{ \overline{ \varphi}} f_n  ||_2 &\leq & || T_{ \overline{ \varphi}} f_n  - T_{ \overline{ \psi}} f_n ||_2 + || T_{ \overline{ \psi}} f_n  ||_2 \\
% & \leq & || \varphi - \psi||_{ \infty} || f_n||_b + || T_{ \overline{ \psi}} f_n  ||_2\\
 %& \leq & \varepsilon C +|| T_{ \overline{ \psi}} f_n  ||_2. 
%\end{eqnarray*}
However, for every $n \in \mathbb{N}$, $|| f_n||_b \leq C $ since $ (f_n)_n$ tends weakly to zero in ${ \mathcal{H} }(b)$.
By the dominated convergence theorem, and using $ \eqref{eq 4}$ it follows that
$$ ||T_{ \overline{ \psi}} f_n  ||_2^2 \leq || \overline{ \psi} f_n ||_2^2 = \int_{ \mathbb{T}} | \psi |^2 | f_n |^2 d \zeta = \int_K | \psi |^2 | f_n |^2 d \zeta \rightarrow 0.$$
whence $ T_{ \overline{ \psi}}:{ \mathcal{H}}(b) \longrightarrow H^2 $ is compact then  $ T_{ \overline{ \varphi}}:{ \mathcal{H}}(b) \longrightarrow H^2 $ is compact
$ ( \Rightarrow )$ Suppose that $\varphi \in  C( \mathbb{T} ), \zeta \in \sigma (b) \cap \mathbb{T} $ and $ {T_{ \overline{ \varphi} }}_{ | { \mathcal{H} }(b) }$ is compact.
Let $$ F_{ \lambda} (z) = \frac{ 1 - | \lambda|^2}{1 - | b( \lambda ) |^2}  \left| \frac{1 - \overline{b( \lambda)} b(z)}{1 - \bar{ \lambda}z} \right|^2,$$\\
which is the absolute value of the normalized reproducing kernel for ${ \mathcal{H} }(b).$ Observe that $ F_{ \lambda} (z) \geq 0 $.\\
Since $ \zeta \in { \sigma}(b) \cap \mathbb{T}$ then there is a sequence $ \lambda_n$ in $ \mathbb{D}$ such that $ \lambda_n \rightarrow \zeta $ and $ | b( \lambda_n ) | \rightarrow c $ with $c < 1$ ( by the definition of the spectrum of $b$ already mentioned).
Suppose that $ \zeta = e^{ i \alpha } $ and note that if $ | t - \alpha | \geq \delta ,$ then \\
$$  F_{ \lambda_n} (e^{it}) \leq C_{ \delta } \frac{1 - | \lambda_n|^2 }{ 1 - | b( \lambda_n )|^2}, $$\\
for some absolute constant $ C_{ \delta } > 0.$ Thus since $ | b( \lambda_n ) | \rightarrow c $ with $c < 1$, we get that 
$$  \sup_{ | t - \alpha| \geq \delta} F_{ \lambda_n} (e^{it}) \rightarrow 0 \hspace{0.2 cm} as \hspace{0.1 cm} n \rightarrow \infty . $$
Write,
\begin{align*}
       & \varphi ( \zeta ) \frac{1}{2 \pi} \int_{ - \pi}^{ \pi} F_{  \lambda_n} (e^{it}) dt - \frac{1}{ \mid \mid k_{ \lambda_n}^b \mid\mid_b^2} < k_{ \lambda_n}^b , T_{ \bar{\varphi} } k_{ \lambda_n}^b >_2 \\
  &= \varphi ( \zeta ) \frac{1}{2 \pi} \int_{ - \pi}^{ \pi} F_{  \lambda_n} (e^{it}) dt - \frac{1}{ \mid \mid k_{ \lambda_n}^b \mid\mid_b^2} < k_{ \lambda_n}^b , P_{+} ( \bar{\varphi} k_{ \lambda_n}^b )>_2\\
 &= \varphi ( \zeta ) \frac{1}{2 \pi} \int_{ - \pi}^{ \pi} F_{  \lambda_n} (e^{it}) dt - \frac{1}{ \mid \mid k_{ \lambda_n}^b \mid\mid_b^2} < k_{ \lambda_n}^b , \bar{\varphi}  k_{ \lambda_n}^b >_2 \hspace{0.3 cm} (\mbox{see \cite[lemma 4.8]{MR3617311}})\\
 & = \varphi ( \zeta ) \frac{1}{2 \pi} \int_{ - \pi}^{ \pi} F_{  \lambda_n} (e^{it}) dt - \frac{1}{ \mid \mid k_{ \lambda_n}^b \mid\mid_b^2} \frac{1}{ 2 \pi} \int_{ - \pi}^{ \pi} \varphi( e^{it} ) | k_{{ \lambda}_n}^b (e^{it})|^2\\
&= \frac{1}{2 \pi} \int_{ -\pi}^{ \pi} \varphi( \zeta) F_{ \lambda_n }( e^{ it} ) dt - \frac{1}{2 \pi} \int_{ -\pi}^{ \pi} \varphi( e^{ it} ) F_{ \lambda_n }( e^{ it} ) dt\\
&= \frac{1}{2 \pi} \int_{ \mathbb{T}} ( \varphi( \zeta) - \varphi( e^{it})    ) F_{ \lambda_n }( e^{it} ) dt\\
&= \frac{1}{2 \pi} \int_{ | t - \alpha | \leq \delta} ( \varphi( \zeta) - \varphi( e^{it})    ) F_{ \lambda_n }( e^{it} ) dt + \frac{1}{2 \pi} \int_{ | t - \alpha | \geq \delta} ( \varphi( \zeta) - \varphi( e^{it})    ) F_{ \lambda_n }( e^{it} ) dt  . 
\end{align*}
The first integral can be made small by the continuity of $\varphi$. Once $\delta >0$ is fixed the second term goes to zero since $  \sup_{ | t - \alpha| \geq \delta} F_{ \lambda_n} (e^{it}) \rightarrow 0 \hspace{0.2 cm} as \hspace{0.1 cm} n \rightarrow \infty . $
In addition
\begin{eqnarray*}
\int_{ -\pi}^{ \pi} F_{ \lambda_n }( e^{ it} ) dt &=& \frac{1 - | \lambda_n|^2}{1 - | b(\lambda_n)|^2} \int_{ -\pi}^{ \pi} \frac{|1 - \overline{ b( \lambda_n)} b(e^{it})|^2}{|1 - \overline{ {\lambda}_n} e^{ it}|^2}dt\\
& \geq & \frac{1 - | \lambda_n|^2}{1 - | b(\lambda_n)|^2} (1 -  |b(\lambda_n)|^2 ) \int_{ -\pi}^{ \pi} \frac{1}{|1 - \overline{ {\lambda}_n} e^{ it}|^2}dt\\
& = &  \frac{(1 - | \lambda_n|^2)(1 - |b(\lambda_n)|^2) }{(1 - |b(\lambda_n)|)(1 + |b(\lambda_n) |)}  \frac{1}{|1 - |{ \lambda}_n|^2}\\
&= & \frac{1 -|b({\lambda}_n)| }{1 +|b({\lambda}_n)|} \geq  \frac{1 - c}{2} >  0.
\end{eqnarray*}
Furthermore, in one hand
$$ \frac{||k_{ \lambda_n}^b ||_2}{||k_{ \lambda_n}^b ||_b} \leq 1.$$
And in the other hand 
the sequence $ (\frac{k_{ \lambda_n}^b}{||k_{ \lambda_n}^b||_b})_n$ converges weakly to $0$, because $| \lambda_n| \rightarrow 1$ and $ | b( \lambda_n ) | \rightarrow c $ with $c < 1$.
 Indeed, using that
$$ ||k_{ \lambda_n}^b||_b^2 = \frac{1 - | b( { \lambda}_n)|^2}{ 1 - |{ \lambda}_n|^2}.$$
We deduce that for $f \in H^{ \infty} \cap \mathcal H(b),$
$$ | < f , \frac{k_{ \lambda_n}^b}{ || k_{ \lambda_n}^b ||_b} >_b | = \frac{ | f( { \lambda}_n)| \sqrt{1 - | \lambda_n|^2 }}{ \sqrt{1 - | b({ \lambda}_n )|^2}} \leq \frac{ || f||_{ \infty} \sqrt{1 - | \lambda_n|^2} }{\sqrt{1 - |b({ \lambda}_n )|^2  }} \rightarrow 0 \hspace{0.5 cm} as \hspace{0.5 cm} n \rightarrow \infty .$$
Furthermore, $ H^{ \infty} \cap { \mathcal{H}}(b) $  is dense ${ \mathcal{H} } (b)$, since for every $\lambda \in \mathbb{D}, \hspace{0.1 cm}  k_{ \lambda}^b \in H^{ \infty} \cap { \mathcal{H}}(b) $.
Then for every $f \in { \mathcal{H}}(b)$  and every $ \varepsilon > 0$, there exists $g \in H^{ \infty} \cap { \mathcal{H}}(b) $ such that $|| f - g ||_b < \frac{ \varepsilon}{2} .$ 
Using previous computations we see that there exists $ N_0 \in \mathbb{N}$ such that $ n \geq N_0 \Rightarrow | < g,\frac{k_{ \lambda_n}^b}{ || k_{ \lambda_n}^b ||_b}>_b | \leq \frac{ \varepsilon}{2}.$
 Hence for $ n \geq N_0$, 
\begin{eqnarray*}
  | < f,\frac{k_{ \lambda_n}^b}{ || k_{ \lambda_n}^b ||_b}>_b | &\leq & | < f-g,\frac{k_{ \lambda_n}^b}{ || k_{ \lambda_n}^b ||_b} >_b | + | < g,\frac{k_{ \lambda_n}^b}{ || k_{ \lambda_n}^b ||_b} >_b |\\
  & \leq & ||f - g||_b + | < g,\frac{k_{ \lambda_n}^b}{ || k_{ \lambda_n}^b ||_b} >_b | \\
  &\leq & \varepsilon. 
\end{eqnarray*}
Thus the sequence $  \frac{k_{ \lambda_n}^b}{ || k_{ \lambda_n}^b ||_b}$ converges weakly to 0 in ${ \mathcal{H}}(b)$, with $ {T_{\overline \varphi}}_{| \mathcal H(b)}$ considered compact.
We deduce that  
\begin{eqnarray*}
 \frac{1}{ \mid \mid k_{ \lambda_n}^b \mid\mid_b^2} | < k_{ \lambda_n}^b , T_{ \bar{\varphi} } k_{ \lambda_n}^b >_2|  &\leq & \frac{||k_{ \lambda_n}^b ||_2}{||k_{ \lambda_n}^b ||_b} \frac{||T_{ \bar{\varphi} }  k_{ \lambda_n}^b ||_2}{||k_{ \lambda_n}^b ||_b}\\
 & \leq & \frac{||T_{ \bar{\varphi} }  k_{ \lambda_n}^b ||_2}{||k_{ \lambda_n}^b ||_b} \leq \frac{||T_{ \bar{\varphi} } k_{ \lambda_n}^b ||_b}{||k_{ \lambda_n}^b ||_b}  \rightarrow 0 \hspace{0.3cm} as \hspace{0.3cm} n \rightarrow \infty.
 \end{eqnarray*}
 After all these computations, we see that 
$$ \varphi ( \zeta ) \frac{1}{2 \pi} \int_{ - \pi}^{ \pi} F_{  \lambda_n} (e^{it}) dt - \frac{1}{ \mid \mid k_{ \lambda_n}^b \mid\mid_b^2} < k_{ \lambda_n}^b , T_{ \bar{\varphi} } k_{ \lambda_n}^b >_b \rightarrow 0 \hspace{ 0.2 cm} as \hspace{0.2 cm} n \rightarrow \infty. $$
with 
$$ \int_{ -\pi}^{ \pi} F_{ \lambda_n }( e^{ it} ) dt > 0 \hspace{0.3 cm} and \hspace{0.3 cm}  \frac{1}{ \mid \mid k_{ \lambda_n}^b \mid\mid_b^2} | < k_{ \lambda_n}^b , T_{ \bar{\varphi} } k_{ \lambda_n}^b >_2| \rightarrow 0 \hspace{0.2 cm} as \hspace{0.2 cm} n \rightarrow \infty. $$
Finally, $\varphi ( \zeta ) = 0 $.\\
\end{proof}
%From this Theorem we only notice a necessary condition concerning the compactness of the operator ${T_{\overline \varphi}}_{| \mathcal H(b)}$ when $b$ is an extreme point and $ \varphi \in  C( \mathbb{T}) \cap H^{ \infty}$. Indeed, if ${T_{\overline \varphi}}_{| \mathcal H(b)}$ is compact then the operator ${T_{\overline \varphi}}: { \mathcal{H} }(b) \rightarrow H^2 $ is compact since it's the composition of two compact operators: the injection $ i : { \mathcal{H} }(b) \hookrightarrow H^2$ and ${T_{\overline \varphi}}_{| \mathcal H(b)}$, thus by Theorem \eqref{thm 2.2}, $ \varphi_{ | { { \sigma}(b) \cap \mathbb{T} } } = 0   .$  While for the sufficient condition, we need one more hypothesis.
%A curious corollary is obtained by noting that, if $b$ is an extreme point and in particular  not inner, the operator ${T_{\overline \varphi}}_{| \mathcal H(b)}$ with $ \varphi \in  C( \mathbb{T}) \cap H^{ \infty}$ is never compact (except for the trivial case where $\varphi = 0$).
We now present consequences of this result.
\begin{corollaire}\label{cor1-extreme}
Let $b$ be an extreme point of the closed unit ball of $ H^{ \infty}$ such that $m( \sigma (b) \cap \mathbb{T}) > 0$. Let $ \varphi \in  C( \mathbb{T}) \cap H^{ \infty}$. Then the operator 
$$\begin{array}{ccccc}
T_{ \overline{ \varphi}}  & : & { \mathcal{H} }(b) & \to & { \mathcal{H} }(b) \\
 & & f & \mapsto & P_+( \overline{ \varphi} f) \\
\end{array},$$\\
is compact if and only if $ \varphi = 0   .$\\
\end{corollaire}
\begin{proof}
Assume that ${T_{\overline \varphi}}_{| \mathcal H(b)}$ is compact. Then $T_{\overline \varphi}: \mathcal{H}(b) \rightarrow H^2 $ is also compact and by Theorem \ref{thm 2.2}, 
$$ \varphi(   { \sigma}(b) \cap \mathbb{T} )  = 0   .$$
However, since $\varphi \in H^{ \infty}$ with $m( \sigma (b) \cap \mathbb{T}) > 0$ then $\varphi = 0$.
\end{proof}
\begin{corollaire}
If $b$ is an extreme point of the closed unit ball of $H^{ \infty}$, which is not an inner function, and if $\varphi \in H^{ \infty} \cap C( \mathbb{T})$, then 
$$\begin{array}{ccccc}
T_{ \overline{ \varphi}}  & : & { \mathcal{H} }(b) & \to & { \mathcal{H} }(b) \\
 & & f & \mapsto & P_+( \overline{ \varphi} f) \\
\end{array},$$
is compact if and only if $ \varphi = 0$.
\end{corollaire}
\begin{proof}
Indeed, Since $b$ is not inner, the set 
$$ E = \{  \zeta \in \mathbb{T} : |b( \zeta)| \neq 1   \}$$
has a positive Lebesgue mesure. Moreover, it turns out that  $ E \subset clos ( \sigma (b)) \cap \mathbb{T}$. Indeed, if $ \zeta \in \mathbb{T} \setminus clos ( \sigma (b)) $ then $b$ admits an analytic continuation across a neighborhood $ D( \zeta,r) = \{ w : |w - \zeta| < r \}$ of $\zeta$ with $|b| \equiv 1$ on the arc $D( \zeta, r) \cap \mathbb{T}$.\\
In particular, $ |b( \zeta)| = 1$ and $ \zeta \in \mathbb{T} \setminus E.$ We deduce that 
$$ 0 < m(E) \leq m(clos ( \sigma (b)) \cap \mathbb{T})).$$ 
To conclude the proof, it remains to apply Corollary~\ref{cor1-extreme}. 
\end{proof}
In the inner case, depending on the spectrum, we could have non trivial compact co-analytic Toeplitz operators or not. 
\begin{exmp}
Let $( { \zeta}_n )_n \subset \mathbb{T}$ be a sequence dense  in $ \mathbb{T}$, and let $ \mu = \sum_{n \geq 1} \frac{1}{n^2} { \delta}_{ {\zeta}_n} $ and $\Theta$ be the singular inner function associated to $\mu$. Then by the definition of the spectrum of $ \Theta$ we have that $ \sigma ( \Theta) = supp( \mu) = \mathbb{T}$.  Hence,  we can apply Theorem~\ref{thm 2.2}
 to get that if $T_{\overline\varphi}:K_\Theta\longrightarrow K_\Theta$ is compact, then $\varphi=0$. 
\end{exmp}
\begin{exmp}
Let $ \Theta (z) = e^{- \frac{1+z}{1-z}}$. Then $ \sigma ( \Theta ) = \{ 1  \} $ and $m( \sigma (\Theta) \cap \mathbb{T})) = 0.$ Let us now consider $ \varphi \in \mathcal{A}( \mathbb{D}) = Hol( \mathbb{D}) \cap C( \overline{ \mathbb{D} }) $ such that $ \varphi \neq 0$ and $ { \varphi}(1)= 0.$ Then by Theorem \ref{thm 2.2} (or Ahern-Clark's result) we know that $T_{ \overline{ \varphi}} : K_{ \Theta} \rightarrow K_{ \Theta}$ is compact.\\
\end{exmp}
We conclude that when $b$ is not an inner fonction, there are no compact Toeplitz operators ${T_{\overline \varphi}}_{| \mathcal H(b)} $ with $\varphi \in H^{ \infty} \cap C( \mathbb{T})$ (expect when $\varphi = 0$).
%We conclude that in the case where $b$ is non-extreme, The Toeplitz operator ${T_{\overline \varphi}}_{| \mathcal H(b)}$ is never compact and for any function $ \varphi \in H^{ \infty}$, while in the extreme ant not inner case we could only show that for $ \varphi \in  C( \mathbb{T}) \cap H^{ \infty}$ there are no compact Toeplitz operators.

\section{ Hypercyclicity of ${T_{\overline \varphi}}_{| \mathcal H(b)}$}

%It is well known that an operator having a large sapply of eigenvectors is hypercyclic. See the Godefroy-Shapiro Criterion (Theorem \eqref{thm 2.4}). Using this result, Godefroy and Shapiro proved that for $ \varphi \in H^{ \infty}$, then $ T_{ \bar{ \varphi }} : H^2 \rightarrow H^2$ is hypercyclic if and only if $ \varphi$ is non constant and $ \varphi ( \mathbb{D}) \cap \mathbb{T} \neq \emptyset $. On the other hand, it is obvious that there are no hypercyclic Toeplitz operators with analytic symbols.\\
%For general symbols $ \varphi \in L^{ \infty} ( \mathbb{T})$, few results are known see a recent paper of A.Baranov and A. Lishanskii \cite{MR3544864} and a paper of Shkarin who studied the case where $ \varphi (z) = a \bar{z} + b + cz.$\\
%In this paper, we are interested in the restriction of $ T_{ \bar{ \varphi}}$ to de Branges-Rovnyak space with symbols $\varphi \in H^{ \infty}$. ( Note that the restriction of co-analytic symbols is necessary if we want that $T_{ \bar{ \varphi }} $ maps $ \mathcal{H}(b)$ into itself).\\
%Using Godefroy-Shapiro Criterion, we succeed to extend their result on Toeplitz operators on $ H^2 $ (see Theorem \eqref{thm 2.6}) in our context of $\mathcal{H}(b) $ spaces, when $ b $ is non extreme.\\

Following the approach of Godefroy-Shapiro, we generalize Theorem \ref{thm 2.6} to operators $T_{\overline{ \varphi}} : \mathcal{H}(b) \rightarrow \mathcal{H}(b) $ when $b$ is a non extreme point of the closed unit ball of $H^{ \infty}$ and $\varphi \in H^{ \infty}.$
Moreover, using a more sophistical criterion due to S.Griveaux we prove that in the non-extreme case, for every $ | \lambda| > 1$ the operator $ \lambda X_b = \lambda S_{| \mathcal{H}(b)}^* $ is not only hypercyclic but frequently hypercyclic.
On the contrary when $b$ is extreme, $ \lambda X_b$ is never hypercyclic.

The following lemma plays a key role in our approach. It generalizes the completeness of the set of all reproducing kernels of $H^2$, $k_{ \lambda}$ for every $ \lambda \in \mathbb{D}$ in $ \mathcal{H}(b)$, when $b$ is non-extreme \cite{MR847333}.
\begin{lemme}\label{lemma 3.1}
Let $b$ be a non-extreme point of the closed unit ball of $ H^{ \infty}$, and let $E$ be a set of $\mathbb{D}$ with an accumulation point in $ \mathbb{D}$. Then $span \{k_{z} : z \in E   \}$ is dense in ${ \mathcal{H}} (b)$.
\end{lemme}
\begin{proof}
Let $ f \in  { \mathcal{ H}}(b)$ such that $f \perp span \{ k_z, z \in E \}$ in $ { \mathcal{ H}}(b)$. Then
 $$ \forall z \in E, \hspace{0.1 cm} <f , k_z  >_b = 0. $$
 Hence
 $$ \forall z \in E, \hspace{0.1 cm} f(z) + \frac{b(z)}{a(z)} f^+ (z) =0, $$ 
 where $f^+ \in H^2 $ is such that $T_{ \bar{b}} f = T_{ \bar{a}} f^+  \hspace{0.1 cm} (see \hspace{0.1cm} Theorem \hspace{0.1cm} \ref{thm 1.2})$.
 We get
 $$ \forall z \in E, \hspace{0.1 cm} a(z) f(z) + b(z) f^+ (z) = 0.$$
 
But since $ af + bf^+ \in Hol( \mathbb{D})$ and $af +bf^+$ vanishes on the set $E$ with an accumulation point in $ \mathbb{D}$, then  $ af + bf^+=0 \hspace{0.1cm} on \hspace{0.1cm} \mathbb{D}$, and this is equivalent to $ af + bf^+  = 0$ on $\mathbb{T}$, multiplying this equality by $ \bar{b}$ and using the identity $ |a|^2 + |b|^2 = 1$ a.e. on $\mathbb{T}$, we obtain
\begin{equation}\label{eq 2}
  a( \bar{b} f - \bar{a} f^+ ) = - f^+ .
  \end{equation}
The relation $ T_{ \bar{b}}f = T_{ \bar{a}} f^+ $ can be written as $ P_+ ( \bar{b} f  - \bar{a}f^+) = 0 $, which means that the function $\bar{b} f  - \bar{a}f^+ \in \overline{ H_0^2}$, with $\overline{ H_0^2} = \{ f \in L^2( \mathbb{T}): \hat{f}(n)=0, n \geq 0   \}$. In particular by $ \eqref{eq 2}$ we deduce that $ f^+ / a \in L^2( \mathbb{T})$. Now, on one hand, it follows that $ f^+ / a \in H^2$  (see \cite[corollary 4.28]{MR3617311}) because $a$ is outer. On the other hand, $ \eqref{eq 2}$ also implies that $ f^+ / a \in \overline{ H_0^2}$, whence $ f^+ / a = 0$. That is $ f^+ = 0 $ and then $f = 0$.
\end{proof}
\begin{theorem}
Let $b$ be a non-extreme point of the closed unit ball of $ H^{ \infty}$ and let $ \varphi \in H^{ \infty}$. Then the operator
  $$\begin{array}{ccccc}
T_{ \bar{ \varphi }}  & : & { \mathcal{H}} (b) & \to & {\mathcal{ H}}(b) \\
 & & f & \mapsto & T_{ \bar{ \varphi }} f = P_+ ( \bar{ \varphi} f ). \\
\end{array}$$\\
 is hypercyclic if and only if $\varphi$ is non-constant and $ \varphi ( \mathbb{D}) \cap \mathbb{T} \neq \emptyset.$
\end{theorem}
\begin{proof}
Since $b$ is non-extreme, the reproducing kernel of $ H^2, k_z$ at $z \in \mathbb{D}$ is an element of ${\mathcal{ H}}(b) \hspace{0.1cm}$. Moreover $ k_{z}$ is an eigenvector of ${T_{ \bar{ \varphi }}}_{\mid {\mathcal{ H}}(b) }$,with associated eigenvalue $\overline{ \varphi( z)} $. Indeed, we have 
$$ <f,  T_{ \bar{ \varphi }} ( k_{ z}) >_2 = < \varphi f, k_z   >_2 = \varphi (z) f(z) =<f, \overline{ \varphi (z) } k_z  >_2  $$\\
for all $ f \in H^2$, so that $ T_{ \bar{ \varphi }} ( k_{ z}) = \overline{ \varphi ( z)} k_{ z}.$ Hence for every $z \in \mathbb{D}$, $ k_z \in \ker ( {T_{ \bar{ \varphi }}}_{\mid {\mathcal{ H}}(b) } - \overline{ \varphi ( z)}).$
Let $ U = \{ z \in \mathbb{D}; | \varphi (z) | < 1 \}$ and $ V = \{ z \in \mathbb{D}; | \varphi (z) | > 1 \}$. If $\varphi$ is non-constant and $ \varphi ( \mathbb{D}) \cap \mathbb{T} \neq \emptyset ,$ the open sets $U$ and $V$ are both non-empty by the open mapping theorem for anaytic functions. In view of Theorem \ref{thm 2.4}, it is enough to notice that since $U$ and $V$ are open sets of $\mathbb{D}$, they have accumulations in $\mathbb{D}$ and by lemma \ref{lemma 3.1}, $  span \{ k_z, z \in U \}$ and $  span \{ k_z, z \in V \}$ are dense in $ { \mathcal{ H}}(b) .$ Therefore $ {T_{ \bar{ \varphi }}}_{\mid {\mathcal{ H}}(b) }$ is hypercyclic.\\

Conversely, assume that $ {T_{ \bar{ \varphi }}}_{\mid {\mathcal{ H}}(b) }$ is hypercyclic ( so that $ \varphi$ is certainly non constant). And assume that $ \varphi ( \mathbb{D}) \cap \mathbb{T} = \emptyset.$ Since $\mathbb{D}$ is connected with $\varphi$ continuous on $\mathbb{D}$, $\varphi( \mathbb{D})$ is connected, hence $ \varphi( \mathbb{D}) \subset \mathbb{D}$ or $ \varphi( \mathbb{D}) \subset \mathbb{C} \setminus \overline{ \mathbb{D}}.$
If $ \varphi( \mathbb{D}) \subset \mathbb{D} $ then $ \forall z \in \mathbb{D}, | \varphi(z)| < 1 $, it implies that $ || \varphi||_{ \infty} \leq 1$ and $ ||{T_{ \bar{ \varphi }}}_{ \mid { \mathcal{ H}}(b) } || \leq || \varphi||_{ \infty} \leq 1$, whence $ {T_{ \bar{ \varphi }}}_{\mid {\mathcal{ H}}(b) } $ is non-hypercyclic (absurd).
If $ \varphi( \mathbb{D}) \subset \mathbb{C} \setminus \mathbb{D} $ then $ \forall z \in \mathbb{D}, | \varphi(z)| > 1 $. In this case, $ \frac{1}{ \varphi} \in H^{ \infty} $ and $ T_{ \frac{1}{ \bar{ \varphi}}} : { \mathcal{ H}}(b)  \longrightarrow { \mathcal{ H}}(b) $ is non-hypercyclic since $||T_{ \frac{1}{ \bar{ \varphi}} }  || \leq || \frac{1}{ \bar{ \varphi}}||_{ \infty} \leq 1.$ Seeing that, $ T_{ \bar{ \varphi}} T_{ \frac{1}{ \bar{ \varphi}}} = T_{ \frac{1}{ \bar{ \varphi}}} T_{ \bar{ \varphi}} = I$, then $ T_{ \frac{1}{ \bar{ \varphi}}} = ( T_{ \bar{ \varphi}})^{-1}$, consequently $ {T_{ \bar{ \varphi }}}_{ \mid { \mathcal{ H}}(b) }$ is non-hypercyclic( indeed an invertible operator is hypercyclic if and only if its inverse is hypercyclic \cite[ page 3]{MR2533318} ). We get also a contradiction.
\end{proof} 
\begin{remarque}
Note that when $b = 0$, we recover Theorem \ref{thm 2.6} of Godefroy and Shapiro.
\end{remarque}
%Using a theorem of S.Grivaux (see Theorem \eqref{thm 2.5}), we can prove that $ \lambda X_b$ is not only hypercyclic but also frequently hypercyclic.\\
In the particular case when $ \varphi (z) = z$, corresponding to operator $X_b : \mathcal{H}(b) \rightarrow \mathcal{H}(b),$ $X_b (f) = S^* f,$ we get a better result using Theorem $\ref{thm 2.5}$.
\begin{theorem}
Let $b$ be a non-extreme point of the closed unit ball of $ H^{ \infty}$. Let the operator
  $$\begin{array}{ccccc}
X_b  & : & { \mathcal{H}} (b) & \to & {\mathcal{ H}}(b) \\
 & & f & \mapsto & S^* f  \\
\end{array}.$$\\
For all $ | \lambda | >1 ,   \lambda X_b$ is frequently hypercyclic (in particular hypercyclic).
\end{theorem}
\begin{proof}
We know that ${ \sigma}_p ( \lambda X_b) = B( 0 , | \lambda|)$ (see \cite[Theorem 24.13]{MR3617311}). The eigenvectors of $ { \lambda} X_b $ associated with eigenvalues  $ \mu \in \mathbb{T}$ are the $ k_{ \overline{ \gamma}},$ with $ \gamma =  \frac{\mu}{\lambda}$. Indeed, $ { \lambda} X_b k_{ \overline{ \gamma}} = \mu k_{ \overline{ \gamma}}.$
Take a sequence $ ( \mu_n )_{ n \geq 1} \subset \mathbb{T}$ which is dense in $ \mathbb{T}$ and without isolated points.
So $ (k_{ \overline{{ \gamma}_n}})_{ n \geq 1 }$ with $ { \gamma}_n = \frac{ {  \mu}_n }{\lambda} $ is a sequence of eigenvectors  of $ \lambda X_b$ associated with eigenvalues $( \mu_n )_n \subset \mathbb{T}$.
%Claim: $ \overline{ vect \lbrace k_{ \overline{{ \gamma}_n}}, n \geq 1\rbrace } = { \mathcal{H}}(b).$\\
Moreover, since the sequence $ \{ \overline{{ \gamma}_n}, n \geq 1 \}$ has an accumulation point on $ \mathbb{D}$ (because $ | \overline{{ \gamma}_n}| = \frac{1}{| \lambda| } < 1),$ lemma \ref{lemma 3.1} implies that $ span \lbrace k_{ \overline{{ \gamma}_n}}, n \geq 1 \rbrace$ is dense in ${ \mathcal{H}}(b) $.
To apply Theorem \ref{thm 2.5}, it remains to show that for every $ \varepsilon > 0$ and every $n \geq 1$, there exists $m \neq n$ such that $||k_{  \overline{{ \gamma}_n}} - k_{ \overline{{ \gamma}_m}} ||_b < \varepsilon$.
 Notice that the map 
$$\begin{array}{ccccc}
 &  & \mathbb{D} & \to & { \mathcal{H}}{(b)} \\
 & & \mu & \mapsto & k_{ \mu}   \\
\end{array},$$ is continuous on $ \mathbb{D}.$\\
Indeed, let $z_0 \in \mathbb{D}$ and let $ ({z_n})_n$ a sequence of $ \mathbb{D}$ such that $ z_n \rightarrow z_0$ on $\mathbb{D}$. For every $ f \in { \mathcal{H}}{(b)}  $, we have \\
$$ <f, k_{z_n}>_b= f(z_n)+ \frac{b(z_n)}{a(z_n)}f^+ (z_n) , $$
where $f^+ \in H^2$ is such that $ T_{ \bar{b}} f = T_{ \bar{a}} f^+$ (see Theorem \ref{thm 1.2}).
Hence 
$$ < f, k_{z_n}>_b \longrightarrow f(z_0)+ \frac{b(z_0)}{a(z_0)}f^+ (z_0) \hspace{0.3 cm} as \hspace{0.2 cm} n \rightarrow  \infty.$$\\
Therefore
$$ < f, k_{z_n}>_b \longrightarrow < f, k_{z_0}>_b \hspace{0.3 cm} as \hspace{0.2 cm} n \rightarrow  \infty.$$
In other words, the sequence $({k_{z_n}})_n$ weakly converges to $ k_{z_0}$ in ${ \mathcal{H}}{(b)} $.\\
\\
On the other hand, using \cite[ Corollary 23.25]{MR3617311}, we also have
$$  ||k_{z_n}||_b = \frac{1}{ 1- |z_n|^2}( 1 + \frac{ |b (z_n)|^2}{ |a(z_n)|^2}) ,    $$
which gives that 
$$ ||k_{z_n}||_b \longrightarrow ||k_{z_0}||_b    \hspace{0.3 cm} as \hspace{0.2 cm} n \rightarrow  \infty. $$
A classical computation in Hilbert space setting shows now that 
$$ || k_{z_n} - k_{z_0}||_b \longrightarrow 0  \hspace{0.3 cm} as \hspace{0.2 cm} n \rightarrow  \infty, $$
which proves the continuity of the map $ \mu \rightarrow k_{ \mu }$ on $ \mathbb{D}$.\\
In particular, for $ \varepsilon > 0$ and $n \geq 1 ,$
\begin{equation}\label{eq 3} 
  \exists \delta_{ \varepsilon} > 0 ; \hspace{0.1 cm} \forall \mu \in \mathbb{D}, \hspace{0.1 cm}| \mu - \overline{{ \gamma}_n} | < { \delta}_{ \varepsilon} \Rightarrow || k_{ \mu} - k_{\overline{{ \gamma}_n}}||_b <\varepsilon.
  \end{equation}
\\
Since $ ({{ \mu}_m})_m$ is dense in $\mathbb{T}$, without isolated points, we can find $m \neq n$ such that $ | \overline{ { \mu}_m} -\overline{ { \mu}_n}  |   \leq \lambda { \delta}_{ \varepsilon},$ which gives by $ \eqref{eq 3}$: $|| k_{ \overline{{ \gamma}_m}} - k_{ \overline{{ \gamma}_n}}||_b < \varepsilon .$ Hence according to Theorem \ref{thm 2.5}, for every $| \lambda| > 1, \lambda X_b$ is frequently hypercyclic.
\end{proof}

As we saw in the previous Theorem, for all $ | \lambda| > 1 ,\lambda X_b$ is hypercyclic, so this naturally raises the question of finding a common hypercyclic vector for $ ( \lambda X_b )_{ | \lambda| > 1 } .$ We will apply Shkarin's Theorem \ref{thm 1.7} but we need to introduce another operator on $ \mathcal{H}(b)$.

It is well known that $ {\mathcal{ H}}(b)$ is invariant under the unilateral forward shift operator $S$ if and only if $b$ is non-extreme \cite[Corollary 20.20]{MR3617311}. In that case, the mapping
$$\begin{array}{ccccc}
S_b  & : & { \mathcal{H}} (b) & \to & {\mathcal{ H}}(b) \\
 & & f & \mapsto & S f= zf  
\end{array}.$$
gives a well-defined operator. Moreover $S_b$ is bounded on ${\mathcal{ H}}(b)$ with $ ||S_b|| =  \sqrt{1+|a(0)|^{-2}}\|S^*b\|_b^2$ (see \cite[Section 24.1]{MR3617311}) In particular, we see that except in the case when $b=0$ (corresponding to $\mathcal H(b)=H^2$), the operator $S_b$ has a norm strictly greater than $1$. 

\begin{theorem}
Let $b$ be a non-extreme point of the closed unit ball of $ H^{ \infty}$, and let 
  $$\begin{array}{ccccc}
X_b  & : & { \mathcal{H}} (b) & \to & {\mathcal{ H}}(b) \\
 & & f & \mapsto & S^* f  
\end{array}.$$
Then, 
$$ \mathcal{G} = \bigcap HC( \lambda X_b ; || S_b|| <  | \lambda | <  \infty ),$$
is a dense $ G_{ \delta}$-set of ${\mathcal{ H}}(b)$.
\end{theorem}
\begin{proof}
We would like to apply Shkarin's Theorem $\ref{thm 1.7}$ with $ a =0$ , $ c = || S_b||^{-1}$ and $ E = \mathcal{P}$, with $ \mathcal{P}$ the set of analytic polynomials, dense in ${\mathcal{ H}}(b)$ \cite[Theorem 23.13]{MR3617311}.
Let 
$$ S_b : { \mathcal{P}} \longrightarrow { \mathcal{P}}; \hspace{0.2cm}  S_b p = zp.$$ 
It is clear that $ X_b S_b = I$.
For all $ 0 < \alpha < \beta < || S_b||^{-1}$, and for all $p \in \mathcal{P}$, we have on one hand $ { \alpha }^{ -n}X_b^n
p \rightarrow 0$ as $ n \rightarrow \infty,$ since from a certain rank $n_0 = deg (p) + 1$, $X_b^{n_0} p = 0$, and on
the other hand, $ || \beta^n S_b^n p    ||_b \leq ( \beta ||S_b|| )^n||p||_b \rightarrow 0$ as $ n \rightarrow \infty.$
Hence, using Theorem \ref{thm 1.7}, we conclude that $ \mathcal{G}$ is a dense $ G_{ \delta}$-set of ${\mathcal{ H}}(b)$. 
\end{proof}
\begin{remarque}
It remains the question of wether we can replace in the previous Theorem the lower bound 
$$ ||S_b|| < | \lambda|  \hspace{0.2 cm} by \hspace{0.2 cm} 1 < | \lambda |.$$
In other word, is the set 
$$\bigcap HC( \lambda X_b ; 1 <  | \lambda | <  \infty )$$
a dense $ G_{ \delta} $-set of ${\mathcal{ H}}(b) ?$ 
\end{remarque}
In the case where $ b $ is extreme, the operator $X_b$ is no longer hypercyclic, which shows a significant difference in the $ {\mathcal{ H}} (b) $ space theory following that $ \log (1- | b |) $ is integrable or not on $ \mathbb{ T} .$ The proof of this result requires basic facts on the spectrum of hypercyclic operators, which we now briefly recall.

Let $X$ be a complex Banach space, and let $T \in \mathcal{L}(X)$ be hypercyclic. Then $ \sigma_p ( T^* ) = \emptyset$ and every connected component of the spectrum of $T$ intersects the unit circle (see \cite[Page 11]{MR2533318}).
\begin{theorem}\label{thm 3.6}
 Let $b$ be an extreme point of the closed unit ball of $ H^{ \infty}$ and let
  $$\begin{array}{ccccc}
X_b  & : & { \mathcal{H}} (b) & \to & {\mathcal{ H}}(b) \\
 & & f & \mapsto & S^* f  
\end{array}.$$
 Then for every complex number $ \lambda, \lambda X_b$ is not hypercyclic.
\end{theorem}
\begin{proof}
For all $ | \lambda | \leq 1 , || \lambda X_b|| \leq 1 $ hence $ \lambda X_b $ is not hypercyclic.
Now take $\lambda \in \mathbb{C}$, $ | \lambda| > 1 $. By \cite[Corollary 26.3]{MR3617311}, we have 
$$ \sigma_p ( \bar{ \lambda} X_b^* ) = \bar{\lambda} \sigma_p ( X_b^*) = \{ \bar{\lambda} \beta : \beta \in \mathbb{D} \hspace{0.2 cm} and \hspace{0.2 cm} b( \beta) = 0 \} \hspace{0.2 cm}.$$\\
By the preceding equality, we notice that, if $b$ has a Blaschke factor then $ \sigma_p ( \bar{ \lambda} X_b^* )  \neq \emptyset$, and thus $ \lambda X_b $ is not hypercyclic. Now if $b$ does not admit a Blaschke factor we get from \cite[Corollary 26.4]{MR3617311} that $\sigma( X_b) \subset \mathbb{T}.$ That implies, since $| \lambda | > 1,$ $\sigma ( \lambda X_b ) \cap \mathbb{T}=  ( \lambda \sigma ( X_b )) \cap \mathbb{T} = \emptyset$. Therefore the connected component of $ \sigma(  \lambda X_b )  $ do not intersect the unit circle, hence $ \lambda X_b $ is not hypercyclic. We conclude that for every complex number $ \lambda, \hspace{0.1cm} \lambda X_b$ is not hypercyclic.\\ 
\end{proof}
In the case when $b$ is extreme, it has not been possible to reach the non-hypercyclicity of the Toeplitz operator $ {T_{ \bar{ \varphi}}}_{|{ \mathcal{ H}} (b)}$. However we give a necessary (but not sufficient) condition for such an operator to be non-hypercyclic.
\begin{proposition}\label{prop 3.7}
Let $b$ be an extreme point of the closed unit ball of $ H^{ \infty}$, and let $ \varphi \in H^{ \infty}$. If the Toeplitz operator $ {T_{ \bar{ \varphi}}}_{|{ \mathcal{ H}} (b)}$ is hypercyclic then ${ \sigma}_p ( {T_{ \bar{ \varphi}}}_{|{ \mathcal{ H}} (b)}) = \emptyset $.
\end{proposition}
\begin{proof}
Assume that there exists $ \lambda \in  { \sigma}_p ( {T_{ \bar{ \varphi}}}_{|{ \mathcal{ H}} (b)})$. Then there exists $f \in { \mathcal{ H}} (b) $, $f \neq 0$ such that $ {T_{ \bar{ \varphi}}}_{|{ \mathcal{ H}} (b)} f = \lambda f. $
Let's take $ g := { \Omega}_b f$, with $ { \Omega}_b : { \mathcal{ H}} (b) \rightarrow { \mathcal{ H}} (b) $ the canonical conjugation  on ${ \mathcal{ H}} (b)$ (i.e. ${ \Omega}_b$ is antilinear, isometric, surjective and ${ \Omega}_b^2 = Id$, see \cite[Section 26.6]{MR3617311}). Moreover, it can be proved that 
$$  {T_{ \bar{ \varphi}}}_{|{ \mathcal{ H}} (b)}^*= { \Omega}_b   {T_{ \bar{ \varphi}}}_{|{ \mathcal{ H}} (b)}     { \Omega}_b,$$  
(see \cite{MR1481605}).

Using this equality, we obtain
\begin{eqnarray*}
 {T_{ \bar{ \varphi}}}_{|{ \mathcal{ H}} (b)}^* g &=& { \Omega}_b   {T_{ \bar{ \varphi}}}_{|{ \mathcal{ H}} (b)} { \Omega}_b  g \\
 &=& { \Omega}_b   {T_{ \bar{ \varphi}}}_{|{ \mathcal{ H}} (b)} { \Omega}_b^2 f\\
 &=&  { \Omega}_b   {T_{ \bar{ \varphi}}}_{|{ \mathcal{ H}} (b)} f \\
 &=& \lambda  { \Omega}_b f = \lambda g.  
\end{eqnarray*}
Hence, $ \lambda \in  { \sigma}_p ( {T_{ \bar{ \varphi}}}_{|{ \mathcal{ H}} (b)}^*) $ which implies that $ { \sigma}_p ( {T_{ \bar{ \varphi}}}_{|{ \mathcal{ H}} (b)}^*) \neq \emptyset$. It contradicts the hypercyclicity of the operator ${T_{ \bar{ \varphi}}}_{|{ \mathcal{ H}} (b)}$.
\end{proof}
\begin{remarque}
It turns out that this condition is not sufficient. Indeed, if $b$ has no Blaschke factor, then $\sigma_p (X_b)$ is empty, though by Theorem \ref{thm 3.6}, we know that $X_b$ is not hypercyclic.
\end{remarque}
%Obviously, this condition is necessary and not sufficient, since by looking at the previous Theorem \eqref{thm 3.6} we notice that the operator $\lambda X_b$ is always not hypercyclic even if $b$ does not admit a Blaschke product, in other words, even if the point spectrum of $\lambda X_b$, defined by: $ { \sigma}_p ( \lambda X_b) = \{ \lambda \bar{\beta} : \beta \in \mathbb{D} , b( \beta) = 0 \}$ is empty (see \cite[Theorem 26.1]{MR3617311}).\\  
\begin{remarque}\label{rem 3.8}
Notice that 
$${ \sigma}_p ( {T_{ \bar{ \varphi}}}_{|{ \mathcal{ H}} (b)}) = \emptyset  \Leftrightarrow  \forall \lambda \in \mathbb{C}, \hspace{0.1 cm}  K_{ ( \varphi - \lambda)_i} \cap { \mathcal{ H}} (b) = \{ 0 \}, $$
with $ \varphi - \lambda = ( \varphi - \lambda)_i ( \varphi - \lambda)_e $, and $( \varphi - \lambda)_i$ and $( \varphi - \lambda)_e$ are respectively the inner and outer part of $\varphi - \lambda$.\\
\end{remarque}
\begin{proof}
Let $\lambda  \in \mathbb{C}$. Then $ \bar{ \lambda} \in  { \sigma}_p ( {T_{ \bar{ \varphi}}}_{|{ \mathcal{ H}} (b)})$ if and only if there exists $f \in { \mathcal{ H}} (b) $, $f \neq 0$ such that 
$$ {T_{ \bar{ \varphi}}} f = \bar{ \lambda} f . $$
This equation is equivalent to
$$  {T_{ \overline{ \varphi - \lambda}}} f = 0  \Leftrightarrow {T_{ (\overline{ \varphi - \lambda})_e  (\overline{ \varphi - \lambda})_i}} f = 0  \Leftrightarrow {T_{ (\overline{ \varphi - \lambda})_i}} f = 0 ,$$
because when $a$ is outer $T_{ \bar{a}}$ is one-to-one.
The last equation is equivalent to 
$$ f \in  K_{ ( \varphi - \lambda)_i} \cap { \mathcal{ H}}(b),$$
where $K_{ ( \varphi - \lambda)_i}$ denotes the model space associated to $( \varphi - \lambda)_i$.
Thus $\bar{ \lambda} \in { \sigma}_p ( {T_{ \bar{ \varphi}}}_{|{ \mathcal{ H}} (b)})$ if and only if $ K_{ ( \varphi - \lambda)_i} \cap { \mathcal{ H}} (b) \neq \{ 0 \} .$
\end{proof}
In particular, if $b$ has a blaschke factor, then ${T_{ \bar{ \varphi}}}_{|{ \mathcal{ H}} (b)}$ is not hypercyclic, as shown in the following result.
\begin{corollaire}
If $ \lambda \in \mathbb{D}$ such that $ b( \lambda) = 0$ then  ${T_{ \bar{ \varphi}}}_{|{ \mathcal{ H}} (b)}$ is not hypercyclic
\end{corollaire}
\begin{proof}
Suppose that $ \lambda \in \mathbb{D}$ such that $ b( { \lambda}) = 0$, then $ k_{ \lambda} = k_{ \lambda}^b \in { \mathcal{ H}} (b) .$ 
Moreover $ T_{ \bar{ \varphi}} k_{ \lambda} = \overline{ \varphi ( { \lambda})} k_{ \lambda}.$ Hence $\overline{ \varphi ( { \lambda})} \in  { \sigma}_p ( {T_{ \bar{ \varphi}}}_{|{ \mathcal{ H}} (b)}).$ Thus by Proposition \ref{prop 3.7} ${T_{ \bar{ \varphi}}}_{|{ \mathcal{ H}} (b)}$ is not hypercyclic.
\end{proof}
\begin{remarque}
If $b$ is extreme and outer, then ${ \sigma}_p ( {T_{ \bar{ \varphi}}}_{|{ \mathcal{ H}} (b)}) = \emptyset.$
\end{remarque}
 \begin{proof}
 Let $ \lambda \in \mathbb{C}$. If ${( \varphi - \lambda)}_i \equiv cte$ then $ K_{{( \varphi - \lambda)}_i} = \{ 0 \}$ ( see \cite[Theorem 18.2]{MR3617311}), which implies by Remark \ref{rem 3.8} that ${ \sigma}_p ( {T_{ \bar{ \varphi}}}_{|{ \mathcal{ H}} (b)}) = \emptyset$.\\

On the other hand if ${( \varphi - \lambda)}_i \neq cte$, and if $ f \in K_{{( \varphi - \lambda)}_i}$  then $f$ is not a cyclic vector for $ S^*$ (since $span \{ S^{*n} f : n \geq 0   \} \subset  K_{{( \varphi - \lambda)}_i}  \neq H^2 $ because $ S^{*} K_{( \varphi - \lambda)_i} \subset K_{( \varphi - \lambda)_i}$, see section \ref{section 2.1} and \cite[Section 1.10]{MR3617311}). Using \cite[Theorem 25.17]{MR3617311}, it implies  $ f \in K_{ \Theta}$ where $ \Theta = b_i$ is the inner part of $b$. But since $b$ is considered outer, then $ b_i \equiv cte$, thus $ K_{ \Theta} = \{  0 \}.$ Hence $K_{ ( \varphi - \lambda)_i} \cap { \mathcal{ H}} (b) = \{ 0 \}$, which also gives that ${ \sigma}_p ( {T_{ \bar{ \varphi}}}_{|{ \mathcal{ H}} (b)}) = \emptyset .$
 
  \end{proof}
\begin{remarque}
In the case where $b$ is extreme and outer, it would be interesting to know if ${T_{ \bar{ \varphi}}}_{|{ \mathcal{ H}} (b)}$ is hypercyclic or not.
\end{remarque}

 \bibliographystyle{plain}

%\bibliography{bibliographie}

\begin{thebibliography}{10}

\bibitem{MR0264385}
P.~R. Ahern and D.~N. Clark.
\newblock On functions orthogonal to invariant subspaces.
\newblock {\em Acta Math.}, 124:191--204, 1970.

\bibitem{MR3544864}
Anton Baranov and Andrei Lishanskii.
\newblock Hypercyclic {T}oeplitz operators.
\newblock {\em Results Math.}, 70(3-4):337--347, 2016.

\bibitem{MR2159459}
Fr\'{e}d\'{e}ric Bayart and Sophie Grivaux.
\newblock Hypercyclicity and unimodular point spectrum.
\newblock {\em J. Funct. Anal.}, 226(2):281--300, 2005.

\bibitem{MR2533318}
Fr\'ed\'eric Bayart and \'Etienne Matheron.
\newblock {\em Dynamics of linear operators}, volume 179 of {\em Cambridge
  Tracts in Mathematics}.
\newblock Cambridge University Press, Cambridge, 2009.

\bibitem{MR0160136}
Arlen Brown and P.~R. Halmos.
\newblock Algebraic properties of {T}oeplitz operators.
\newblock {\em J. Reine Angew. Math.}, 213:89--102, 1963/1964.

\bibitem{MR0244795}
Louis de~Branges and James Rovnyak.
\newblock Canonical models in quantum scattering theory.
\newblock In {\em Perturbation {T}heory and its {A}pplications in {Q}uantum
  {M}echanics ({P}roc. {A}dv. {S}em. {M}ath. {R}es. {C}enter, {U}.{S}. {A}rmy,
  {T}heoret. {C}hem. {I}nst., {U}niv. of {W}isconsin, {M}adison, {W}is.,
  1965)}, pages 295--392. Wiley, New York, 1966.

\bibitem{MR0215065}
Louis de~Branges and James Rovnyak.
\newblock {\em Square summable power series}.
\newblock Holt, Rinehart and Winston, New York-Toronto, Ont.-London, 1966.

\bibitem{MR3617311}
Emmanuel Fricain and Javad Mashreghi.
\newblock {\em The theory of {$ \mathcal{H}(b)$} spaces. {V}ol. 2}, volume~21
  of {\em New Mathematical Monographs}.
\newblock Cambridge University Press, Cambridge, 2016.

\bibitem{MR3203060}
Stephan~Ramon Garcia, William~T. Ross, and Warren~R. Wogen.
\newblock {$C^*$}-algebras generated by truncated {T}oeplitz operators.
\newblock In {\em Concrete operators, spectral theory, operators in harmonic
  analysis and approximation}, volume 236 of {\em Oper. Theory Adv. Appl.},
  pages 181--192. Birkh\"auser/Springer, Basel, 2014.

\bibitem{MR884467}
Robert~M. Gethner and Joel~H. Shapiro.
\newblock Universal vectors for operators on spaces of holomorphic functions.
\newblock {\em Proc. Amer. Math. Soc.}, 100(2):281--288, 1987.

\bibitem{MR1111569}
Gilles Godefroy and Joel~H. Shapiro.
\newblock Operators with dense, invariant, cyclic vector manifolds.
\newblock {\em J. Funct. Anal.}, 98(2):229--269, 1991.

\bibitem{MR2975340}
Sophie Grivaux.
\newblock A new class of frequently hypercyclic operators.
\newblock {\em Indiana Univ. Math. J.}, 60(4):1177--1201, 2011.

\bibitem{MR1133377}
B.~A. Lotto and D.~Sarason.
\newblock Multiplicative structure of de {B}ranges's spaces.
\newblock {\em Rev. Mat. Iberoamericana}, 7(2):183--220, 1991.

\bibitem{MR847333}
Donald Sarason.
\newblock Doubly shift-invariant spaces in {$H^2$}.
\newblock {\em J. Operator Theory}, 16(1):75--97, 1986.

\bibitem{MR1289670}
Donald Sarason.
\newblock {\em Sub-{H}ardy {H}ilbert spaces in the unit disk}, volume~10 of
  {\em University of Arkansas Lecture Notes in the Mathematical Sciences}.
\newblock John Wiley \& Sons, Inc., New York, 1994.
\newblock A Wiley-Interscience Publication.

\bibitem{MR2363975}
Donald Sarason.
\newblock Algebraic properties of truncated {T}oeplitz operators.
\newblock {\em Oper. Matrices}, 1(4):491--526, 2007.

\bibitem{MR2557957}
Stanislav Shkarin.
\newblock Remarks on common hypercyclic vectors.
\newblock {\em J. Funct. Anal.}, 258(1):132--160, 2010.

\bibitem{MR1481605}
Daniel Su\'{a}rez.
\newblock Backward shift invariant spaces in {$H^2$}.
\newblock {\em Indiana Univ. Math. J.}, 46(2):593--619, 1997.

\end{thebibliography}

\end{document}